\documentclass[12pt, reqno]{amsart}
\usepackage{amsmath, amsthm, amscd, amsfonts, amssymb, graphicx, color}
\usepackage[bookmarksnumbered, colorlinks, plainpages]{hyperref}
\hypersetup{colorlinks=true,linkcolor=red, anchorcolor=green, citecolor=cyan, urlcolor=red, filecolor=magenta, pdftoolbar=true}

\newtheorem{theorem}{Theorem}[section]
\newtheorem{lemma}[theorem]{Lemma}
\newtheorem{proposition}[theorem]{Proposition}
\newtheorem{corollary}[theorem]{Corollary}
\theoremstyle{definition}
\newtheorem{definition}[theorem]{Definition}
\newtheorem{example}[theorem]{Example}

\theoremstyle{remark}
\newtheorem{remark}[theorem]{Remark}

\numberwithin{equation}{section}

\begin{document}
	
	\setcounter{page}{1}
	
	\title[Controlled finite continuous frames]{Controlled finite continuous frames}

	\author[H. Massit, M. Rossafi, C.  Park]{Hafida Massit$^{1}$, Mohamed Rossafi$^{2}$ and Choonkil Park$^{3*}$}
	
	\address{$^{1}$Department of Mathematics, Faculty Of Sciences, University of Ibn Tofail, Kenitra, Morocco}
	\email{\textcolor[rgb]{0.00,0.00,0.84}{massithafida@yahoo.fr}}
	\address{$^{2}$LaSMA Laboratory,  Department of Mathematics,  Faculty of Sciences, Dhar El Mahraz University Sidi Mohamed Ben Abdellah, Fes, Morocco}
	\email{\textcolor[rgb]{0.00,0.00,0.84}{rossafimohamed@gmail.com; mohamed.rossafi@usmba.ac.ma}}
	\address{$^{3}$Research Institute for Natural Sciences, Hanyang University, Seoul 04763, Republic of Korea}
	\email{\textcolor[rgb]{0.00,0.00,0.84}{baak@hanyang.ac.kr}}

	
	
	\date{%10/03/2020; %Accepted: zzzzzz.
		\newline \indent $^{*}$Corresponding author: Choonkil Park (email: baak@hanyang.ac.kr, fax: +82-2-2281-0019, orcid: 0000-0001-6329-8228).}

	\begin{abstract} In this  paper, we present controlled finite continuous frames in a finite dimensional Hilbert space and  we study some properties of them. Parseval controlled integral frames are presented and we characterize  operators that construct controlled integral finite frames. 
	\end{abstract}
	\maketitle

	\section{Introduction and preliminaries}
	
The concept of frames in Hilbert spaces has been introduced by Duffin and Schaffer \cite{Duf} in 1952 to study some deep problems in nonharmonic Fourier series, after the fundamental paper \cite{DGM} by Daubechies, Grossman and Meyer, frame theory began to be widely used, particularly in the more specialized context of wavelet frames and Gabor frames. The majority of these applications requires frames in finite-dimensional spaces. For example, Jamali et al \cite{H} and Javanshiri et al \cite{JAN}, were obtained results that are interesting in applications of frames.

Recently, controlled frames were introduced by Balzas \cite{PDA}, Antoine and Grybos to improve the  numerical efficiency of iterative algorithms for inverting the frame operator on abstract Hilbert spaces \cite{BAL}, however they are used earlier in \cite{bo} for spherical wavelets. For more details, the reader can refer to  \cite{PDA, gx, kvs,  Musazadeh}. 

The concept of a generalization of frames to a family indexed by some locally compact space endowed with a Radon measure was proposed by Kaisar \cite{KAI} and independently by Ali, Antoine and Gazeau \cite{GAZ}. In this paper we try to give a generalization of the results given in \cite{Z} moving from the discrete case to the continuous case.

  For more information on frame theory and its applications, we refer the readers to \cite{FR1, 5, r6, r1, RFDCA}.

	Throughout this paper, assume that $(\mathfrak{A},\mu)$ is a measure space with positive measure $\mu$, $\mathcal{H}$ and $ \mathcal{H}^{N} $ are used for showing a Hilbert space and a finite-dimensional Hilbert space, respectively,  $ GL(H) $ denotes the set of all bounded linear operators with a bounded inverse, and  $ GL^{+}(H)$ is the set of positive operators in $ GL(H) $.

	\begin{definition} \cite{5} 
		Let $\mathcal{H}^{N} $ be an $ N$-dimentional Hilbert space, and $(\mathfrak{A},\mu)$ be a measure space. Then 
		a map $ F : \mathfrak{A} \rightarrow \mathcal{H}^{N} $ is  called an integral frame in $ \mathcal{H}^{N}  $ if    there exist $ 0<A\leq B < \infty $ such that
					\begin{equation}\label{1}
				A\|f\|^{2}\leq \int_{\mathfrak{A}} \langle f, F_{\varsigma}\rangle \langle F_{\varsigma},f \rangle d\mu(\varsigma)\leq B\|f\|^{2} \quad \forall f\in \mathcal{H}^{N} .
			\end{equation}
	The elements $ A $ and $ B $ are called the integral frame bounds.
	If $ A = B $, we call this an integral tight frame.
	If $ A = B = 1 $, it is called an integral Parseval frame.
	If only the right hand inequality of (\ref{1}) is satisfied, we call $  F $ a controlled
	integral Bessel map with bound $ B $.
	\end{definition}

If $ F $ is a Bessel map, then  $ T_{F}: L^{2}(\mathfrak{A},\mu) \rightarrow \mathcal{H}^{N}$, defined by 
$ T_{F}(f)= \int_{\mathfrak{A}} \langle f, F(\varsigma) \rangle F(\varsigma) d\mu(\varsigma)$,
is a bounded linear operator. 
$ T_{F} $ is surjective and bounded if and only if $ F $ is an integral frame.
 This operator is called the synthesis operator.
 
 The adjoint of $ T_{F} $, which  is called the analysis operator, is defined by 
 \begin{equation*}
 	T^{\ast}_{F}:\mathcal{H}^{N} \rightarrow L^{2}(\mathfrak{A}, \mu), \;\;\;\;T^{\ast}_{F}(f)(\varsigma)=\langle f,F(\varsigma)\rangle , \; \varsigma\in \mathfrak{A}.
 \end{equation*}

The continuous frame operator is defined to be $ S_{F}= T_{F}T^{\ast}_{F} $, it is invertible and positive. 

Recall that a Bessel map $ F $ is a frame if and only if there exists a continuous Bessel  mapping  $ G $ is a dual of F if for any $ f,g\in \mathcal{H}^{N} $ 
\begin{equation*}
	\langle f,g\rangle= \int_{\mathfrak{A}} \langle f,G(\varsigma)\rangle \langle g,F(\varsigma)\rangle d\mu(\varsigma), \;f,g\in \mathcal{H}^{N},
\end{equation*}

$ G $ is called a dual frame for $ F $ and $ S^{-1} _{F}F$ is a dual of $ F $.

\section{Main results}

We consider some properties of controlled continuous frames in finite Hilbert spaces.

	\begin{definition}  
	Let $\mathcal{H}^{N} $ be an $ N$-dimensional Hilbert space and $(\mathfrak{A},\mu)$ be a measure space. Then 
a  family $ \{F_{\varsigma}\}_{\varsigma\in \mathfrak{A}} $ is called  a $ V$-controlled integral frame for an invertible operator $ V $ on $ \mathcal{H}^{N} $ if there exist $ 0<A\leq B < \infty $ such that
	\begin{equation}\label{d1eq1}
		A\|f\|^{2}\leq \int_{\mathfrak{A}} \langle f, F_{\varsigma}\rangle \langle VF_{\varsigma},f \rangle d\mu(\varsigma)\leq B\|f\|^{2} \quad \forall f\in \mathcal{H}^{N}.
	\end{equation}
	The elements $ A $ and $ B $ are called the $ V$-controlled integral frame bounds.
	If $ A = B $, we call this a $ V $-controlled integral tight frame.
	If $ A = B = 1 $, it is  called a $ V $-controlled integral Parseval frame.
	If only the right hand inequality of (\ref{d1eq1}) is satisfied, we call $  F $ a $ V $-controlled
	integral Bessel map with bound $ B $.
	
	Similar to ordinary frames, the controlled integral frame operator is defined for a controlled frame on $ \mathcal{H}^{N} $ by $ S_{VF}f= \int_{\mathfrak{A}}\langle f,F_{\varsigma}\rangle VF_{\varsigma} d\mu(\varsigma) $.Wihch assumed in weak sense.
	
	The controlled synthesis operator $ T^{\ast}_{VF}: L^{2}(\mathfrak{A},\mu) \rightarrow \mathcal{H}^{N} $  is defined by 	
	$T T^{\ast}_{VF}(f)= \int_{\mathfrak{A}} \langle f, F(\varsigma) \rangle VF(\varsigma) d\mu(\varsigma) $ and $ S_{UF}= T^{\ast}_{VF}T_{F} $,  where $ T_{F} $ is the analysis operator of $ \{F_{\varsigma}\}_{\varsigma} $.
\end{definition}

\begin{proposition}\label{p0}
Let $ F: \mathfrak{A} \rightarrow \mathcal{H}^{N} $ such that $ \int_{\mathfrak{A}}\Vert VF((\varsigma)) \Vert^{2} d\mu(\varsigma) < \infty$. Then $V F $ is a Bessel map.
\end{proposition}

\begin{proof}
	Using Cauchy-Schwarz inequality, we have
	\begin{equation*}
		\int_{\mathfrak{A}} \vert \langle f,VF(\varsigma)\rangle \vert^{2} d\mu(\varsigma)\leq \int_{\mathfrak{A}}\Vert f\Vert ^{2} \Vert VF(\varsigma) \Vert^{2}d\mu(\varsigma) \leq B \Vert f\Vert^{2}\; with\; B= \int_{\mathfrak{A}}\Vert VF((\varsigma)) \Vert^{2} d\mu(\varsigma).
	\end{equation*}
This completes the proof.	
\end{proof}

We  prove that the converse of Proposition \ref{p0} holds if $ \mathcal{H}^{N} $ is finite dimensional.

\begin{proposition} Let $\mathcal{H}^{N}$ be an $ N$-dimensional Hilbert space and $ F: \mathfrak{A} \rightarrow \mathcal{H}^{N} $ be a Bessel map. Then $ \int_{\mathfrak{A}}\Vert VF((\varsigma)) \Vert^{2} d\mu(\varsigma) < \infty $.
\end{proposition}

\begin{proof}Let $ \{e_{k}\}_{k\in\{1,2,...,n\}} $ be an orthonormal basis for $ \mathcal{H}^{N} $. 
Then  we  have $ \Vert VF(\varsigma) \Vert ^{2} =\sum_{k=1}^{ n}\vert \langle VF(\varsigma) , e_{k}\rangle \vert^{2}$.
So 
	\begin{align*}
		\int_{\mathfrak{A}}\Vert VF(\varsigma) \Vert ^{2} d\mu(\varsigma)&=\sum_{k=1}^{ n}\int_{\mathfrak{A}}\vert \langle VF(\varsigma) , e_{k}\rangle \vert^{2}d\mu(\varsigma)\\
		&\leq \sum_{k=1}^{ n} B \Vert e_{k} \Vert ^{2}= Bn<\infty.
	\end{align*}
This completes the proof.
\end{proof}

We give a new identity for controlled integral frames in finite dimensional Hilbert spaces.

\begin{proposition} \label{p1}Let $ \{F_{\varsigma}\}_{\varsigma\in \mathfrak{A}} $  be a $ V $-controlled integral frame where $ V $ is an invertible operator on $ \mathcal{H}^{N} $. Then the following statements are equivalent.
\begin{itemize}
	\item [(1)] $ \{F_{\varsigma}\}_{\varsigma \in \mathcal{A}} $ is a $ V- $controlled integral frame with bounds $ A $ and $ B $.
	\item[(2)] 	
$S_{VF}(f)= \int_{\mathfrak{A}}\langle f,F_{\varsigma}\rangle VF_{\varsigma} d\mu(\varsigma)$  
is an invertible and positive operator on $ \mathcal{H}^{N} $.
\end{itemize}	
\end{proposition}

\begin{proof}$ (1) \Rightarrow (2)$ is immediately from the definition of $ V- $ controlled integral frame operator.
	
 $ (2) \Rightarrow (1)$  for any $ f\in \mathcal{H}^{N} $, suppose that $ S_{VF} $ is positive and invertible.

 Then 
 \begin{equation*}
 	\langle S_{VF}f,f\rangle = \langle \int_{\mathfrak{A}}\langle f,F_{\varsigma}\rangle  VF_{\varsigma}f,f  d\mu(\varsigma)\rangle= \int_{\mathfrak{A}}\langle f,F_{\varsigma}\rangle \langle VF_{\varsigma}f,f \rangle d\mu(\varsigma).
 \end{equation*}
 This implies that
 \begin{equation*}
  \Vert\int_{\mathfrak{A}}\langle f,F_{\varsigma}\rangle \langle VF_{\varsigma}f,f \rangle d\mu(\varsigma)\Vert= \Vert\langle S_{VF}f,f\rangle\Vert= \Vert S_{VF}^{\frac{1}{2}}f\Vert^{2} ,
 \end{equation*} 
 there exists $ 0<m $ such that 
 \begin{equation}\label{2}
 	m\langle f,f\rangle  \leq \langle S_{VF}f,f\rangle .
 \end{equation}
On other hand, for all $ f\in \mathcal{H} $, there exists $ 0<m' $ such that 

\begin{equation}\label{1}
\langle S_{VF}f,f\rangle \leq m' \langle f,f\rangle
\end{equation}

%  $ \langle S_{VF}f,f\rangle \leq \Vert S_{VF}\Vert \langle If,f \rangle  \Rightarrow S_{VF} \leq \Vert S_{VF}\Vert I  $ and 	$  S_{VF}^{-1} \leq \Vert S_{VF}^{-1}\Vert I  $.
 From \ref{2} and \ref{1}, we conclude that  $ \{F_{\varsigma}\}_{\varsigma \in \mathcal{A}} $ is a $ V- $controlled integral frame
 % Therefore 
 %\begin{equation*}
 %%	\Vert S_{VF}^{-1}\Vert ^{-1}I\leq S_{VF} \leq \Vert S_{VF}\Vert I
 %	\end{equation*}
% which  complete the proof.
\end{proof}

\begin{theorem} \label{t3.4} Let  $ \{F_{\varsigma}\}_{\varsigma\in \mathfrak{A}} $ be a  continuous frame with the frame operator $ S_{F} $. If $ V \in GL^{+}(\mathcal{H}^{N})$ is self-adjoint operator  with $V S_{F}=S_{F}V $, then $ \{F_{\varsigma}\}_{\varsigma\in \mathfrak{A}} $ is a $ V-$controlled integral frame.
\end{theorem}

\begin{proof}For $ f\in \mathcal{H} $, we have 
	\begin{equation*}
		\langle S_{V}f,f\rangle =\langle \int_{\mathfrak{A}}\langle f,F_{\varsigma}\rangle  VF_{\varsigma}f,f  d\mu(\varsigma)\rangle= \int_{\mathfrak{A}}\langle f,F_{\varsigma}\rangle \langle VF_{\varsigma}f,f \rangle d\mu(\varsigma).
	\end{equation*} 
So, we have 
\begin{equation*}
	A\Vert f\Vert^{2}\leq \langle S_{V}f,f\rangle.
\end{equation*}
Then, The operator $ S_{VF} $ is positive, also it's selfadjoint.
	Let $ S_{VF}=VS_{F} $. The operator $ S_{VF} $ is invertible. By Proposition \ref{p1}, 
$ \{F_{\varsigma} \}_{\varsigma}$ is a $ V- $controlled integral frame. 	
\end{proof} 

\begin{proposition}Let  $ \{F_{\varsigma}\}_{\varsigma\in \mathfrak{A}} $ be a $ V-$controlled integral frame for $ \mathcal{H} $ and $ V \in GL(\mathcal{H})$. Then  $ \{F_{\varsigma}\}_{\varsigma} $ is a continuous frame and  $V S_{F}=S_{F}V $, with 
	\begin{equation*}
	  \int_{\mathfrak{A}}\langle f,F_{\varsigma}\rangle  VF_{\varsigma} d\mu(\varsigma) =\int_{\mathfrak{A}}\langle f,VF_{\varsigma}\rangle  F_{\varsigma} d\mu(\varsigma) .  
	\end{equation*}	
\end{proposition}
\begin{proof}
	Let  $ \{F_{\varsigma}\}_{\varsigma \in \mathcal{A}} $ is a $ V- $controlled integral frame with bounds $ A $ and $ B $.
	
	We have 
	\begin{equation*}
		A\langle f,f\rangle \leq \langle S_{V}f,f\rangle = \langle VS f,f\rangle =\langle V^{\frac{1}{2}} Sf,V^{\frac{1}{2}}f\rangle\leq \Vert V^{\frac{1}{2}} \Vert^{2}\langle Sf,f\rangle.
	\end{equation*}
So, 
\begin{equation}\label{3}
A\Vert V^{\frac{1}{2}}\Vert^{-2}  \langle f,f\rangle  \leq	\int_{\mathfrak{A}}\langle f,F_{\varsigma}\rangle \langle F_{\varsigma},f\rangle d\mu(\varsigma).
\end{equation}
On other hand, for all $ f\in\mathcal{H} $ we have
\begin{align*}
	\int_{\mathfrak{A}}\langle f,F_{\varsigma}\rangle \langle F_{\varsigma},f\rangle d\mu(\varsigma)&= \langle Sf,f\rangle\\
&	= \langle V^{-1}V Sf,f\rangle \\
&	=\langle (V^{-1}VS)^{\frac{1}{2}}f,(V^{-1}VS)^{\frac{1}{2}}f\rangle \\
&	\leq \Vert V^{\frac{-1}{2}} \Vert^{2}\langle (VS)^{\frac{1}{2}}f,(VS)^{\frac{1}{2}}f\rangle \\
&=\Vert V^{\frac{-1}{2}} \Vert^{2}\langle (S_{V})^{\frac{1}{2}}f,(S_{V})^{\frac{1}{2}}f\rangle \\
&= \Vert V^{\frac{-1}{2}} \Vert^{2}\langle S_{V}f,f\rangle \\
&\leq \Vert V^{\frac{-1}{2}} \Vert^{2} B\langle f,f\rangle. 
\end{align*}
Then, 
\begin{equation}\label{4}
\int_{\mathfrak{A}}\langle f,F_{\varsigma}\rangle \langle F_{\varsigma},f\rangle d\mu(\varsigma)\leq \Vert V^{\frac{-1}{2}} \Vert^{2} B\langle f,f\rangle.
\end{equation} 
From \ref{3} and \ref{4} we conclude that $  \{F_{\varsigma}\}$ is a continuous frame.
\end{proof}

We show that the condition $ VS_{F}= S_{F}V $ is not given in a finite-dimensional real Hilbert space in the following example. 

\begin{example}
 Consider the frame $ \{F_{\varsigma}\}_{\varsigma} = \{\binom{1}{\varsigma}\}$ for $ \mathbb{R}^{2}$ and  
$\mathfrak{A} = [0,1] $ endewed with the Lebesgue measure.	
With the operator $ V= 	 \begin{pmatrix}
	\begin{array}{cc}
		1 & 1 \\ 
	-1 & 1
	\end{array} 
\end{pmatrix}, $ 
it is clear that $ V $ is positive and invertible. By definition of frame operator, we have $ S_{F}= 	 \begin{pmatrix}
	\begin{array}{cc}
		1 & 1/2 \\ 
		-1/2 & 1/3
	\end{array} 
\end{pmatrix}.$

For  all $  f =\binom{x}{y} \in \mathbb{R}^{2}$, we have 
\begin{align*}
	\int_{\mathfrak{A}} \langle f, F_{\varsigma}\rangle \langle VF_{\varsigma},f \rangle d\mu(\varsigma)&=\langle S_{VF}f,f\rangle\\
	&=\langle VS_{F}\binom{x}{y},\binom{x}{y}\rangle\\
	&= \dfrac{1}{2}x^{2}- \dfrac{y^{2}}{2}-\dfrac{2}{3}xy.
\end{align*}
	
We obtain that the frame $ \{F_{\varsigma}\} $ is a $ V-$controlled integral frame such that $ VS_{F}\not= S_{F}V $.
\end{example}

\begin{proposition}\label{p9}
If $ \{F_{\varsigma}\}_{\varsigma\in \mathfrak{A}} $ is  a $ V-$controlled integral frame for $ \mathcal{H}^{N} $  with the frame operator $ S_{VF} $, then   $ \{F_{\varsigma}\}_{\varsigma\in \mathfrak{A}} $ is a continuous frame for $ \mathcal{H} $  with the frame operator $ V^{-1}S_{VF} $.
\end{proposition}

\begin{proof}Let $ S_{F}f= \int_{\mathfrak{A}}  \langle f, F_{\varsigma}\rangle F_{\varsigma} d\mu(\varsigma), \forall f\in\mathcal{H}^{N} $ and 
	$ S_{VF}f= \int_{\mathfrak{A}}  \langle f, F_{\varsigma}\rangle VF_{\varsigma} d\mu(\varsigma) = VS_{F}f, \forall f\in\mathcal{H}^{N} $. Then $ V^{-1}S_{VF}f= S_{F}f $. The operator $ S_{F} $ 	is an  injective operator on finite dimensional Hilbert space and  $ S_{F} $  is invertible. Therefore, for $ f\in \mathcal{H}^{N}  $ we have
	\begin{equation*}
		f= \int_{\mathfrak{A}}  \langle S_{F}^{-1}f, F_{\varsigma}\rangle F_{\varsigma} d\mu(\varsigma)=\int_{\mathfrak{A}}  \langle f,(S_{F}^{-1})^{\ast} F_{\varsigma}\rangle F_{\varsigma} d\mu(\varsigma).
	\end{equation*}
This shows that $ \{F_{\varsigma}\}_{\varsigma\in \mathfrak{A}} $ is a continuous frame  generator for $ \mathcal{H}^{N} $
with the frame operator $ V^{-1}S_{VF} $ and $ \{F_{\varsigma}\}_{\varsigma\in \mathfrak{A}} $ is a generator for $ \mathcal{H}^{N} $.
\end{proof}

%%%%%%%%%%%%%%%%%
 
\begin{theorem}Let $ F=\{F_{\varsigma}\}_{\varsigma\in \mathfrak{A}} $ be a $ V-$controlled integral frame for $ \mathcal{H}^{N} $ where $ V $ is an invertible operator and the controlled frame operator $ S_{VF} $ be a normal operator with $ VS_{VF}= S_{VF}V $. Then $ V $ is a positive operator. 
\end{theorem}

\begin{proof} By Proposition \ref{p9},  $ \{F_{\varsigma}\}_{\varsigma\in \mathfrak{A}} $ is a continuous frame with the frame operator $ S_{F}=V^{-1}S_{VF} $. 	
	We have $ S_{VF}V=V S_{VF}$ and so $ S_{VF}S_{F}= VS_{F}S_{F}= S_{F}VS_{F}= S_{F}S_{VF}$.
		There exists a set of common orthonormal eigenvectors of $ S_{VF} $ and $ S_{F} $ as  $ \{e_{k}\}_{k\in\{1,2,...,N\}} $.
	
	Let $ \{\alpha_{k}\}_{k\in\{1,2,...,N\}} $ and $ \{\beta_{k}\}_{k\in\{1,2,...,N\}} $ be  eigenvalues of operators $ S_{VF} $ and $ S_{F} $, respectively.
	
	For $ \varsigma\in\mathfrak{A} $,  we have
	\begin{equation*}
		Ve_{k}= ( S_{VF} S_{F}^{-1}) (e_{k})= S_{VF}(\beta_{k}^{-1}e_{k})=\beta_{k}^{-1} \alpha_{k} e_{k}.
	\end{equation*}
Then 
\begin{equation*} 
	Vf= \sum_{k=1}^{N} \beta_{k}^{-1} \alpha_{k} \langle f,e_{k}\rangle 
\end{equation*}
Which follows $ V $ is a positive operator.
\end{proof}

\begin{proposition} Let $ F=\{F_{\varsigma}\}_{\varsigma\in \mathfrak{A}} $ be a continuous frame for $ \mathcal{H}^{N} $ with the frame operator $ S_{F} $. If $ \{e_{k}\} _{k\in\{1,2,...,N\}}$  and $ \{\beta_{k}\} _{k\in\{1,2,...,N\}}$ are the set of orthonormal eigenvectors and the set of eigenvalues of $ S_{F} $, respectively, then for every set   $ \{\alpha_{k}\} _{k\in\{1,2,...,N\}} \subseteq (0,+\infty)$,  $ F=\{F_{\varsigma}\}_{\varsigma\in \mathfrak{A}} $ is a $ V- $controlled frame, where $ V $ is defined by  $ V e_{k}= \alpha_{k}e_{k}$, for $ k= 1,\cdots,N $.	
\end{proposition}

\begin{proof}Let $ f\in \mathcal{H}^{N} $. Then  we have	
	\begin{align*}
		VS_{F}f=VS_{F}(\sum_{k=1}^{N} \langle f,e_{k} \rangle e_{k}) )&= V(\sum_{k=1}^{N} \langle f,e_{k} \rangle S_{F}e_{k})\\
		&=\sum_{k=1}^{N}\alpha_{k} \langle f,e_{k} \rangle  Ve_{k}\\
		&= \sum_{k=1}^{N}\beta_{k} \langle f,e_{k} \rangle  \alpha_{k}e_{k}\\
		&= \sum_{k=1}^{N}\alpha_{k} \langle f,e_{k} \rangle  S_{F}e_{k}\\
		&= S_{F}\sum_{k=1}^{N} \langle f,e_{k} \rangle  \alpha_{k}e_{k}\\
		&= S_{F}\sum_{k=1}^{N} \langle f,e_{k} \rangle  Ve_{k}\\
		&= S_{F}Vf.
	\end{align*}
Since $ \{\alpha_{k}\}\subset (0,\infty) $, so 	
 $ V $  is positive and invertible and also $ V $ and $ VS_{F} $ commute with each other. So  $ VS_{F} $ is an invertible and positive operator with
	\begin{equation*}
		VS_{F}f=V (\sum_{k=1}^{N} \langle f,F_{k}\rangle F_{k} )=\sum_{k=1}^{N}\langle f,F_{k}\rangle VF_{k} .
	\end{equation*} 
Therefore, $ \{F_{\varsigma}\} _{\varsigma}$ is a $ V- $controlled integral frame with the frame operator $ VS_{F} $.
\end{proof}

\begin{theorem}Let $ F=\{F_{\varsigma}\}_{\varsigma\in \mathfrak{A}} $ be a $ V- $ controlled integral frame with frame operator $ S_{VF} $ and $ L\in GL(\mathcal{H}^{N}) $ ($ L $ is positive and so it is self -adjoint) such that $ LV =VL $. Then $ \{LF_{\varsigma}\}_{\varsigma \in\mathfrak{A}} $ is a $ V- $controlled frame with frame operator $ LS_{VF} L^{\ast}$. Moreover, $ \{L^{k}F_{\varsigma}\}_{\varsigma \in\mathfrak{A}} $ is a $ V- $ controlled integral frame for $ k\in\mathbb{R} $ with frame operator $ L^{k}S_{VF}(L^{k})^{\ast} $. 
	
\end{theorem} 

\begin{proof}We have 
	\begin{equation*}
		S_{VLF}(f)= \int_{\mathfrak{A}}\langle f,LF_{\varsigma}\rangle VLF_{\varsigma} d\mu(\varsigma)= \int_{\mathfrak{A}} \langle f,LF_{\varsigma}\rangle LVF_{\varsigma} d\mu(\varsigma)= LS_{VF}L^{\ast}f.
	\end{equation*}
	Thus $ S_{VLF}= LS_{VF}L^{\ast} $  is invertible and
	\begin{equation*}
	\langle LS_{VF}L^{\ast}f,f\rangle = \langle S_{VF}L^{\ast},L^{\ast}f \geq0,\;\forall f\in\mathcal{H}^{N}, \; i.e.,\; S_{VF}\geq 0 .
	\end{equation*}
	This gives that $ S_{VLF} $ is positive. Hence $ \{LF_{\varsigma}\}_{\varsigma} $ is a $ V- $ controlled integral frame. 
	Also we have $L^k V = V l^k$.  Thus  $ \{L^{k}F_{\varsigma}\}_{\varsigma} $ is a $ V-$controlled integral frame with the frame operator  $ L^{k}S_{VF}(L^{k})^{\ast} $.  	
\end{proof}

\begin{corollary}If $ \{F_{\varsigma}\}_{\varsigma} $ is a $ V- $ controlled integral frame such that $ VS_{F} =S_{F}V$, then $ \{ S_{F}^{\frac{\beta-1}{2}}F_{\varsigma}\} _{\varsigma}$ is a $ V- $controlled integral frame for any $ \beta \in\mathbb{R} $, with frame operator $ VS_{F} $.	
\end{corollary}

  Gramian operator or Gramian matrix  for a $ V- $controlled frame has been defined in \cite{Z} and we introduce Gramian operator for $ V- $controlled frame in finite Hilbert spaces, and we consider its properties.

\begin{definition} Let $ \{F_{\varsigma}\}_{\varsigma} $ be a $ V- $controlled integral frame with analysis operator $ T_{F} $ and synthesis operator $ T^{\ast}_{VF} $. Then  the operator  $ G_{VF}: = T_{F}T^{\ast}_{VF} $  is called a $ V-$Gramian operator. The canonical matrix representation of Gramian operator of a $ V-$controlled integral frame $ \{F_{\varsigma}\} $ is obtained by 
	\begin{equation*}
		G= (\langle V F_{i},F_{j}\rangle )_{i,j\in\mathfrak{A}}.
	\end{equation*}
%	\begin{center} 
	
%$  	$$ \begin{pmatrix}
	
	%	\langle VF_{1},F_{1}\rangle & 	\langle VF_{2},F_{1}\rangle  & \dots &	\langle VF_{N},F_{1}\rangle \\ 
	%	. & .& \dots & .\\
	%	\vdots & \vdots & \vdots & \vdots\\
	%		\langle VF_{1},F_{N}\rangle & 	\langle VF_{2},F_{N}\rangle  & \dots &	\langle VF_{N},F_{N}\rangle
	 
%\end{pmatrix} $$ $
%\end{center}
\end{definition}

The following theorem investigates the Gramian matrix of the transferred $ V- $controlled integral frames.

\begin{theorem} Let $ \{F_{\varsigma}\}_{\varsigma} $ be a $ V-$controlled integral frame for $ \mathcal{H}^{N} $ and $ T $ be a linear operator that commutes with $ V $. Then $T$ is unitary if and only if the $ V-$Gramian matrix of $ \{TF_{\varsigma}\} $ is equal to $ G_{VF} $.
	\end{theorem}

\begin{proof} Suppose that   $ T $ is unitary. Then we have	
	\begin{equation*}
		G_{V(TF)}=\{\langle TVF_{\beta},TF_{\alpha}\rangle \}_{\beta,\alpha}=\{\langle VF_{\beta},F_{\alpha} \rangle\}= G_{VF}.
	\end{equation*}

Conversely, let $ G_{VF}=G_{V(TF)} $. Then 
\begin{equation*}
	\langle VTF_{\beta},TF_{\alpha}\rangle = \langle VF_{\beta},F_{\alpha}\rangle 
\end{equation*}
and 
\begin{equation*}
	\langle T^{\ast}VTF_{\beta} -VF_{\beta}, F_{\alpha}\rangle=0 .
\end{equation*}
For $ f\in\mathcal{H}^{N} $, we have 
\begin{equation*}
	f=\int_{\mathfrak{A}} \langle f,S_{F}^{-1}F_{\varsigma} \rangle d\mu(\varsigma).
\end{equation*} 
Then
\begin{align*}
  	(T^{\ast}VT-V)f& = (T^{\ast}VT-V) \int_{\mathfrak{A}} \langle f,S_{F}^{-1}F_{\varsigma} \rangle F_{\varsigma} d\mu(\varsigma)\\
  	& =  \int_{\mathfrak{A}} \langle f,S_{F}^{-1}F_{\varsigma} \rangle (T^{\ast}VT-V)  F_{\varsigma} d\mu(\varsigma)\\
  	 & =0 . 
\end{align*}
Thus we have  $ T^{\ast}VT= V $ and $ T^{\ast}T=I $.
\end{proof}
%%%%%%%%%%%%%%%%%%%%%%%%%%%%%%%%%%%%%%%%%%%
%%%%%%%%%%%%%%%%%%%%%%%%%%%%%%%%%%%%%%%%%%%

Parseval frames are the closest family to orthonormal bases. We present and study some properties of Parseval controlled integral frames in a finite-dimensional Hilbert space.

\begin{lemma}\cite{GAV}\label{l.12} Let $\mathcal{H}^{N}$ be an $ N- $dimensional Hilbert space, and $G, L: \mathfrak{A} \rightarrow \mathcal{H}^{N} $ be continuous Parseval frames and $ K \in L( \mathcal{H}^{N}) $ be self-adjoint. Then 
	\begin{equation*}
		\int_{\mathfrak{A}} \Vert K G(\varsigma)\Vert ^{2}d\mu(\varsigma)=	\int_{\mathfrak{A}} \Vert K L(\varsigma)\Vert ^{2}d\mu(\varsigma).
	\end{equation*}	
\end{lemma}

\begin{theorem} \label{t}Let $\mathcal{H}^{N}$ be an $ N- $dimensional Hilbert space, and $ F $ be a frame for  $\mathcal{H}^{N}$ and $ G $ be a Parseval  $ V- $ controlled integral frame for  $\mathcal{H}^{N}$.	
	 Then
	\begin{eqnarray*}
	\int_{\mathfrak{A}}\Vert V (G(\varsigma)-F(\varsigma))\Vert^{2}d\mu(\varsigma) &=& \int_{\mathfrak{A}}\Vert VS^{-1/2}F(\varsigma)-VF(\varsigma)\Vert^{2}d\mu(\varsigma)\\ &+& \int_{\mathfrak{A}}\Vert V(S^{1/4}G(\varsigma)- S^{-1/4}F(\varsigma))\Vert^{2} d\mu(\varsigma).
	\end{eqnarray*} 	
\end{theorem}

\begin{proof}
By Lemma \ref{l.12} with $ L(\varsigma)= VS^{1/2}F(\varsigma) $, we have 
	\begin{equation*}
		\int_{\mathfrak{A}}\Vert G(\varsigma)\Vert^{2} d\mu(\varsigma)= \int_{\mathfrak{A}}\Vert VS^{-1/2}F(\varsigma)\Vert^{2} d\mu(\varsigma)
	\end{equation*} 
	and 
		\begin{equation*}
		\int_{\mathfrak{A}}\Vert S^{1/4} G(\varsigma)\Vert^{2} d\mu(\varsigma)= \int_{\mathfrak{A}}\Vert S^{1/4}F(\varsigma)\Vert^{2} d\mu(\varsigma)=\int_{\mathfrak{A}} \Vert S ^{-1/4}F(\varsigma)\Vert^{2}d\mu(\varsigma).
	\end{equation*}
Thus 
\begin{align*}
&	\int_{\mathfrak{A}}\Vert G(\varsigma)-F(\varsigma)\Vert^{2}d\mu(\varsigma) -\int_{\mathfrak{A}}\Vert VS^{-1/2}F(\varsigma)-VF(\varsigma)\Vert^{2}d\mu(\varsigma)\\
	&=-2Re \int_{\mathfrak{A}}\langle G(\varsigma),F(\varsigma)\rangle d\mu(\varsigma)+2\int_{\mathfrak{A}}\langle  VS^{-1/2}F(\varsigma),F(\varsigma)\rangle d\mu(\varsigma)\\
	&=-2Re\int_{\mathfrak{A}}\langle  VS^{1/4}G(\varsigma),S^{-1/4}F(\varsigma)\rangle d\mu(\varsigma)+\int_{\mathfrak{A}} \Vert VS^{-1/4}F(\varsigma)\Vert^{2} d\mu(\varsigma)+\int_{\mathfrak{A}}\Vert VS^{-1/4}G(\varsigma)\Vert^{2} d\mu(\varsigma)\\
	&= \int_{\mathfrak{A}}\Vert V(S^{1/4}G(\varsigma)- S^{-1/4}F(\varsigma))\Vert^{2} d\mu(\varsigma).
\end{align*} 
This completes the proof.
\end{proof}

\begin{corollary}Let $\mathcal{H}^{N}$ be an $ N-$dimensional Hilbert space and $ F $ be a frame for  $\mathcal{H}^{N}$ with frame operator $ S $. For every  Parseval  $ V-$controlled integral frame $ G $   of $\mathcal{H}^{N}$, we have 
	\begin{equation*}
		\int_{\mathfrak{A}}\Vert V (G(\varsigma)-F(\varsigma))\Vert^{2}d\mu(\varsigma)\geq \int_{\mathfrak{A}}\Vert VS^{-1/2}F(\varsigma)-VF(\varsigma)\Vert^{2}d\mu(\varsigma)
	\end{equation*} 
and we have equality if and only if 
\begin{equation*}
 G(\varsigma)=VS^{-1/2}F(\varsigma)\;,\varsigma\in\mathfrak{A}.
\end{equation*}
 
\end{corollary}
\begin{proof}The first part follows immediately from Theorem \ref{t}.
	
	We have equality if and only if 
	\begin{align*}
		&	S^{1/4}G(\varsigma)= VS^{-1/4}F(\varsigma)\;,\varsigma\in\mathfrak{A}\\
		& \Leftrightarrow G(\varsigma)=VS^{-1/2}F(\varsigma).
	\end{align*}
This completes the proof.
\end{proof}
%%%%%%%%%%%%%%%%%%%%%%%%%%%%%

\begin{proposition}Let $ F=\{F_{\varsigma}\}_{\varsigma\in \mathfrak{A}} $ be a $ V- $ controlled integral frame for $ \mathcal{H}^{N} $. Then
	\begin{equation*}
		\int_{\mathfrak{A}} \langle VF_{\varsigma},F_{\varsigma}\rangle d\mu(\varsigma)= N.
	\end{equation*} 
\end{proposition}

\begin{proof} Let $ \{e_{k}\}_{k=1}^{N} $be an orthonormal basis for $ \mathcal{H}^{N} $. 	
	We have 
	\begin{equation*}
		e_{k}=S_{VF}e_{k}= \int_{\mathfrak{A}}\langle e_{k},F_{\varsigma}\rangle VF_{\varsigma}d\mu(\varsigma).
	\end{equation*} 
	Thus
	\begin{align*}
		N= \sum_{k=1}^{N}\Vert e_{k}\Vert^{2}&=\sum_{k=1}^{N} \int_{\mathfrak{A}} \langle e_{k},F_{\varsigma}\rangle \langle VF_{\varsigma}, e_{k}\rangle d\mu(\varsigma)\\
	&=\int_{\mathfrak{A}}\sum_{k=1}^{N}\langle e_{k},F_{\varsigma}\rangle \langle VF_{\varsigma}, e_{k}\rangle d\mu(\varsigma)\\
	&=\int_{\mathfrak{A}} \langle VF_{\varsigma}, F_{\varsigma}\rangle d\mu(\varsigma).
	\end{align*}
This completes the proof.
\end{proof}

%%%%%%%%%%%%%%%%%%%%%%%%%%%%%%%%%

The following proposition illustrates that the orthogonal projections can be preserved controlled frames in a finite-dimensional Hilbert space.

\begin{proposition} Let $ \{F_{\varsigma}\}_{\varsigma \in\mathfrak{A}} $ be a $ V- $ controlled integral frame for  $ \mathcal{H}^{N} $,  $ E $ be a subspace of $ \mathcal{H}^{N} $ and $ U $ be an orthonogonal projection of $ \mathcal{H}^{N} $ onto $ E $ such that $ VU = UV $. Then $ \{UF_{\varsigma}\}_{\varsigma} $ is a $ V- $controlled frame for $ E $.
		 If  $ \{F_{\varsigma}\}_{\varsigma \in\mathfrak{A}} $ is a Parseval $ V- $controlled integral frame for $ \mathcal{H}^{N} $, then $ \{UF_{\varsigma} \}_{\varsigma}$ is a Parseval $ V- $controlled integral frame for $ E $.	
\end{proposition}

\begin{proof}
	For all $ f\in E $, we have
	\begin{equation*}
		A\Vert f\Vert^{2}= A\Vert Uf\Vert^{2} \leq \int_{\mathfrak{A}}\langle Uf,F_{\varsigma}\rangle \langle VF_{\varsigma}, Uf \rangle d\mu(\varsigma) \leq B \Vert Uf\Vert^{2}= B\Vert f\Vert^{2}
	\end{equation*}
and 
	\begin{equation*}
	A\Vert f\Vert^{2} \leq \int_{\mathfrak{A}}\langle f,UF_{\varsigma}\rangle \langle UVF_{\varsigma}, f \rangle d\mu(\varsigma) \leq  B\Vert f\Vert^{2},
\end{equation*}
which implies that 
\begin{equation*}
	A\Vert f\Vert^{2} \leq \int_{\mathfrak{A}}\langle f,UF_{\varsigma}\rangle \langle VUF_{\varsigma}, f \rangle d\mu(\varsigma) \leq  B\Vert f\Vert^{2}.	
\end{equation*}
Therefore,  $ \{UF_{\varsigma}\}_{\varsigma} $ is a $ V- $ controlled integral frame for $ E $. 

Suppose $ \{F_{\varsigma}\}_{\varsigma} $ is a Parseval $ V- $controlled integral frame.
 Then for every $ f\in E $, 
\begin{align*}
	S_{VUF}(f)&=\int_{\mathfrak{A}}\langle f,UF_{\varsigma}\rangle VUF_{\varsigma} d\mu(\varsigma)\\
	&= \int_{\mathfrak{A}}\langle Uf,F_{\varsigma}\rangle UVF_{\varsigma} d\mu(\varsigma)\\
	&=U\int_{\mathfrak{A}}\langle Uf,F_{\varsigma}\rangle VF_{\varsigma} d\mu(\varsigma)\\
	&=U^{2}f\\
	&=f.
\end{align*} 
Therefore, $ \{UF_{\varsigma}\} _{\varsigma}$ is a Parseval $ V- $controlled integral frame for $ E $.
\end{proof}

If $ \{F_{\varsigma}\}_{\varsigma\in\mathfrak{A}} $ is a $ V- $controlled integral frame with the controlled frame operator $ S_{VF} $, then $ S _{VF}= VS_{F} $ and $ f= \int_{\mathfrak{A}}\langle f,F_{\varsigma}\rangle (S_{VF}^{-1}V)fd\mu(\varsigma) $ for every $ f\in  \mathcal{H}^{N} $. This gives that  $ \{F_{\varsigma}\}_{\varsigma\in\mathfrak{A}} $ is a Parseval $ S_{VF}^{-1}V- $controlled integral frame, and $ \{S_{VF}^{-1}VF_{\varsigma}\}_{\varsigma\in\mathfrak{A}} $.

\begin{theorem}Let $ \{F_{\varsigma}\}_{\varsigma\in\mathfrak{A}} $ be a continuous frame with the frame operator $ S_{F} $. 
Then  every tight controlled integral frame $ \{F_{\varsigma}\}_{\varsigma\in\mathfrak{A}} $ is exactly an $ \alpha- $tight 
	$ \alpha S_{F}^{-1} $ controlled integral frame for $ \alpha \in \mathbb{C} $.	
\end{theorem}

\begin{proof}Let $ \{F_{\varsigma}\}_{\varsigma\in\mathfrak{A}} $ be an $ \alpha- $tight $ V- $controlled integral frame, for $ \alpha \in\mathbb{C} $. Then for $ f\in \mathcal{H}^{N} $, $\alpha f=\int_{\mathfrak{A}} \langle f,F_{\varsigma}\rangle VF_{\varsigma}d\mu(\varsigma)  $ and so $ \alpha I= S_{VF} = VS_{F}$ and $ V= \alpha S^{-1}_{F} $, i.e., $ \{F_{\varsigma}\}_{\varsigma\in\mathfrak{A}} $ is a  $ \alpha- $tight 
	$ \alpha S_{F}^{-1} $ controlled integral frame. 	
\end{proof}

%%%%%%%%%%%%%%%%%%%%%%%%%%%%%%

We need to recall properties of the trace of linear operators on  $  \mathcal{H}^{N} $ and then consider trace of an operator by controlled integral frames.

The trace of a linear operator $ L\in L(\mathcal{H}^{N}) $ is defined by 
\begin{equation*}
	Tr(L)= \sum_{k=1}^{N}\langle Le_{k},e_{k}\rangle,
\end{equation*}
where $ \{e_{k}\}_{k=1}^{N} $ is an orthonormal basis for $  \mathcal{H}^{N} $. If $ L _{1}$ and $ L_{2} $ are self-adjoint positive operators, then $ 0\leq Tr(L_{1}L_{2})  \leq Tr(L_{1}) \cdot Tr(L_{2})$.

\begin{proposition}Let $ \{F_{\varsigma}\}_{\varsigma\in\mathfrak{A}} $ be a $ V- $controlled integral frame such that $ V\in GL(\mathcal{H}^{N}) $ is  a self-adjoint operator. Then 
	\begin{equation*}
		Tr(S_{VF})\leq Tr(V)\int_{\mathfrak{A}} \Vert F_{\varsigma}\Vert ^{2}d\mu(\varsigma).
	\end{equation*}
	
\end{proposition}
\begin{proof}Let $ \{\alpha_{k}\}_{k=1}^{N} $ be  the set of eigenvalues of the operator frame $ S_{F} $. Then $ Tr(S_{F})= \sum_{k=1}^{N}\alpha_{k} = \int_{\mathfrak{A}}\Vert F_{\varsigma}\Vert ^{2} d\mu(\varsigma)$.
	
Therefore, 
\begin{equation*}
	Tr(S_{VF})= Tr(VS_{F})\leq Tr(V)Tr(S_{F})= Tr(V)\sum_{k=1}^{N}\alpha_{k}= Tr(V)\int_{\mathfrak{A}}\Vert F_{\varsigma}\Vert ^{2}d\mu(\varsigma).
\end{equation*}
This completes the proof.	
\end{proof}

\begin{proposition}Let $ \{F_{\varsigma}\}_{\varsigma \in \mathfrak{A}} $ be a Parseval $ V- $controlled integral frame for $ \mathcal{H}^{N} $ and $ G  $ be a linear operator on $ \mathcal{H}^{N} $. Then $ Tr(G)= \int_{\mathfrak{A}} \langle GVF_{\varsigma},F_{\varsigma} \rangle d\mu(\varsigma)$. 	
\end{proposition}

\begin{proof}Let  $ \{e_{k}\}_{k=1}^{N} $ be  an orthonormal basis for $  \mathcal{H}^{N} $. Then $ Tr(G)= \sum_{k=1}^{N}\langle Ge_{k},e_{k}\rangle  $. Since $ \{F_{\varsigma}\}_{\varsigma} $ is a Parseval $ V- $controlled integral frame, for $ k\in\{1,2,3,...\} $, we have 
	\begin{equation*}
		Ge_{k}= \int_{\mathfrak{A}} \langle Ge_{k},F_{\varsigma}\rangle V F_{\varsigma}d\mu(\varsigma)
	\end{equation*}
	and 
	\begin{align*}
		Tr(G)&= \sum_{k=1}^{N}\langle  \int_{\mathfrak{A}} \langle Ge_{k},F_{\varsigma}\rangle  VF_{\varsigma}, e_{k}\rangle d\mu(\varsigma)\\
		&=\sum_{k=1}^{N} \int_{\mathfrak{A}} \langle e_{k},G^{\ast}F_{\varsigma}\rangle \langle VF_{\varsigma},e_{k}\rangle d\mu(\varsigma)\\
		&= \int_{\mathfrak{A}} \langle \sum_{k=1}^{N}\langle VF_{\varsigma},e_{k}\rangle e_{k},G^{\ast}F_{\varsigma}\rangle d\mu(\varsigma)\\
		&= \int_{\mathfrak{A}}\langle VF_{\varsigma},G^{\ast}F_{\varsigma}\rangle d\mu(\varsigma)\\
		&=\int_{\mathfrak{A}}\langle GVF_{\varsigma},F_{\varsigma}\rangle d\mu(\varsigma).
	\end{align*}
This completes the proof.
\end{proof}

\begin{remark}
	Every $ \alpha- $ tight $ V- $controlled integral frame $\{ F_{\varsigma}\}_{\varsigma} $ induces a Parseval controlled integral frame.
		We have for every $ f\in\mathcal{H}^{N} $
	\begin{equation*}
		\alpha f= \int_{\mathfrak{A}} \langle f,F_{\varsigma}\rangle VF_{\varsigma}d\mu(\varsigma),
	\end{equation*} 
	and then
	
	\begin{equation*}
		f= \int_{\mathfrak{A}} \langle f,F_{\varsigma}\rangle (\alpha^{-1}V)F_{\varsigma} d\mu(\varsigma).
	\end{equation*}
This means that $ \{F_{\varsigma}\}_{\varsigma} $ is a Parseval $ \alpha^{-1}V -$ controlled frame.

Also, $ \{F_{\varsigma}\} $ is equivalent to $ \{(\alpha^{-1}V)F_{\varsigma}\}_{\varsigma} $ and $ \{(\alpha^{-1}V)F_{\varsigma}\}_{\varsigma} $ is a dual for $ \{F_{\varsigma}\} $.
We recall that a frame $ \{F_{\varsigma} \}_{\varsigma}$ is equivalent to a frame $ \{G_{\varsigma}\}_{\varsigma} $ if there exists an invertible operator $ P \in B(\mathcal{H}^{N}) $ such that $ F_{\varsigma}=P G_{\varsigma} $.
\end{remark}

\begin{example}
	Consider the Hilbert space $\L^{2}( \mathbb{R})$ and  
	$\mathfrak{A} = [0,1] $ endewed with the Lebesgue measure.	
	Let 	$ F=\{F_{\varsigma}\}_{\varsigma} $ and $ G=\{G_{\varsigma}\}_{\varsigma} $ with	
	$ F_{\varsigma}(x_{1},x_{2},x_{3},...)= \dfrac{x_{\varsigma}}{3} $ and $ G_{\varsigma}(x_{1},x_{2},x_{3},...)= \dfrac{x_{\varsigma}}{2} $. Then 
	we obtain
	\begin{equation*}
		\int_{\mathfrak{A}} \Vert F_{\varsigma}(x) \Vert^{2} d\mu(\varsigma) = \dfrac{1}{9}\Vert x\Vert^{2}
	\end{equation*}
	and 
	\begin{equation*}
		\int_{\mathfrak{A}} \Vert G_{\varsigma}(x) \Vert^{2} d\mu(\varsigma) = \dfrac{1}{4}\Vert x\Vert^{2}.
	\end{equation*}
	Define $ P: \mathcal{H}\rightarrow \mathcal{H} $ by $ P(x_{1},x_{2},x_{3},...)= (\dfrac{2}{3}x_{1},\dfrac{2}{3}x_{2},\dfrac{2}{3}x_{3},...)$. 	
	Then $ F_{\varsigma}P= G_{\varsigma}$ for all $ \varsigma \in \mathfrak{A} $.
\end{example}

\begin{theorem} Let $ \{F_{\varsigma}\}_{\varsigma \in \mathfrak{A}} $ be a $ V- $controlled integral frame. Then  $ \{F_{\varsigma}\}_{\varsigma \in \mathfrak{A}} $  has a dual frame that is equivalent to 
$ \{F_{\varsigma}\}_{\varsigma \in \mathfrak{A}} $.
\end{theorem}be

 \begin{proof} Let $ S_{VG} $ be  the frame operator of $ \{F_{\varsigma}\}_{\varsigma} $.
 	 	Then for $ f\in\mathcal{H}^{N} $, 
 	\begin{equation*}
 		S_{VG}f= \int_{\mathfrak{A}} \langle f,F_{\varsigma}\rangle VF_{\varsigma}. d\mu(\varsigma)
 	\end{equation*} 
 	Since $ S_{VG} $ is invertible, we have 
 	\begin{equation*}
 		f= \int_{\mathfrak{A}} \langle f, F_{\varsigma}\rangle (S^{-1}_{VG}V)F_{\varsigma}d\mu(\varsigma).
 	\end{equation*} 
 	Thus  $ \{F_{\varsigma}\}_{\varsigma} $ is a Parseval controlled integral frame and the frame $ \{S_{VF}^{-1}VF_{\varsigma}\}_{\varsigma} $ is a dual frame for $ \{F_{\varsigma}\}_{\varsigma} $ such that it is equivalent to $ \{F_{\varsigma}\}_{\varsigma} $.
 \end{proof}

\begin{proposition}
	If $ \{G_{\varsigma}\}_{\varsigma} $ is a dual of $ \{F_{\varsigma}\}_{\varsigma}$ such that it is equivalent to $ \{F_{\varsigma}\}_{\varsigma} $, then   $ \{G_{\varsigma}\}_{\varsigma} $ induces a Parseval controlled integral frame of  $ \{F_{\varsigma}\}_{\varsigma} $. 
\end{proposition}

\begin{proof} Let  $ \{G_{\varsigma}\}_{\varsigma} $ be a  dual of  $ \{F_{\varsigma}\}_{\varsigma} $
	such that it is equivalent to  $ \{F_{\varsigma}\}_{\varsigma} $. Then there exists an  invertible operator $ V $ such that $ G_{\varsigma}=VF^{\varsigma} $ for every $ \varsigma \in\mathfrak{A} $.
	We have 
	\begin{equation*}
		f=\int_{\mathfrak{A}} \langle f,F_{\varsigma} \rangle G_{\varsigma} d\mu(\varsigma),\;\forall f\in\mathcal{H}^{N}.
	\end{equation*}
Then  $ \{F_{\varsigma}\}_{\varsigma} $ is a Parseval $ V- $controlled integral frame.
\end{proof}
\section{Conclusion}
In this manuscript we introduced and characterized controlled finite continuous frames particularly Parseval controlled finite continuous frames as a subset of dual frames and we reviewed some notions and properties of operators and frames in Hilbert spaces. Also, we defined controlled finite continuous frames and  we gave their properties. Gramian matrix and its properties for controlled finite continuous frames are examined.
In the end we studied controlled finite continuous frames as a proper subset of dual frames is presented by the equivalent frames. 

We will apply these results in a future work in Hardy and Sobolev spaces.
\medskip

\section*{Declarations}

\medskip

\noindent \textbf{Availablity of data and materials}\newline
\noindent Not applicable.

\medskip

\noindent \textbf{Competing  interest}\newline
\noindent The authors declare that they have no competing interests.

\medskip

\noindent \textbf{Fundings}\newline
\noindent  Authors declare that there is no funding available for this article.

\medskip

\noindent \textbf{Authors' contributions}\newline
\noindent The authors equally conceived of the study, participated in its
design and coordination, drafted the manuscript, participated in the
sequence alignment, and read and approved the final manuscript. 

\medskip

%%%%%%%%%%%%%%%%%%%%%%%%%

\end{document}